\documentclass[12pt,a4paper]{article}
\usepackage[active]{srcltx}
\usepackage{amsfonts}
\usepackage{amsmath}
\usepackage{graphicx,epsfig}
\usepackage{amssymb}
\usepackage{theorem,subfigure}
\usepackage[usenames,dvipsnames]{color} 
\usepackage{colortbl}
\definecolor{gray73}{RGB}{186,186,186}
\usepackage{graphicx}
\usepackage{array}
\usepackage{tikz}
\usetikzlibrary{matrix,arrows}
\usepackage{verbatim}
\usepackage{geometry}
\usepackage{multirow}

\usepackage{capt-of}
\usepackage{caption}
\DeclareCaptionLabelFormat{mylabelformat}{#1 #2:}
\newtheorem{theorem}{Theorem}[section]
\newtheorem{lemma}{Lemma}[section]
\newtheorem{remark}[theorem]{Remark}

\newtheorem{algorithm}[theorem]{Algorithm}

\newenvironment{proof}[1][Proof]{\noindent\textbf{#1:} }{\ \rule{0.5em}{0.5em}}

\numberwithin{equation}{section}

\newcommand{\C}{\mbox{$\mathbb C$}}

\newcommand{\supp}{\mbox{supp}\,}
\renewcommand{\mod}{\mbox{mod}\,}
\DeclareMathOperator{\ii}{\mathrm i\!}
\DeclareMathOperator{\e}{\mathrm e}
\font\mfett=cmmib10 at12pt
\def\eps{\hbox{\mfett\char034}}
\begin{document}
\title{A deterministic sparse FFT algorithm for vectors with small support}
\author {Gerlind Plonka\footnote{University of G\"ottingen, Institute for Numerical and Applied Mathematics, Lotzestr.\ 16-18,  37083 G\"ottingen, Germany. Email: plonka@math.uni-goettingen.de} \qquad
Katrin Wannenwetsch\footnote{University of G\"ottingen, Institute for Numerical and Applied Mathematics, Lotzestr.\ 16-18,  37083 G\"ottingen, Germany. Email: k.wannenwetsch@math.uni-goettingen.de}}

\maketitle

\abstract{In this paper we consider the special case where a signal $\mathbf x\in \C^N$ is known to vanish outside a support interval of length $m < N$. If the support length  $m$ of $\mathbf x$ or a good bound of it is a-priori known we derive a sublinear deterministic  algorithm to compute $\mathbf x$ from its discrete Fourier transform $\widehat{\mathbf x}\in\C^N$. 
In case of exact Fourier measurements we require only  ${\cal O}(m \log m)$  arithmetical operations. For noisy measurements, we propose 
 a stable ${\cal O}(m \log N)$ algorithm.
 }
\smallskip

\noindent
\textbf{Key words.} discrete Fourier transform, sparse Fourier reconstruction, sublinear sparse FFT \\
\textbf{AMS Subject classifications.} 65T50, 42A38

\section{Introduction}

It is well-known that  FFT algorithms  for the computation of the discrete Fourier transform of length $N$ 
require ${\cal O}(N \log N)$ arithmetical operations.
However, if the resulting vector is a-priori known to be sparse, i.e., contains only a small number of 
non-zero components, the question arises, whether we can do this computation in an even faster way.
In this paper we derive a new deterministic sparse FFT algorithm for signals $\mathbf x\in \C^N$ that are  a-priori known to vanish outside a support interval of length $m < N$.
\smallskip

{\bf Related work.}
In recent years many sublinear algorithms for the sparse Fourier transform have been proposed that focus on (approximately) $m$-sparse vectors or vectors being norm bounded. 

Most of the these methods are randomized  poly-logarithmic sparse Fourier algorithms achieving e.g.\ a complexity  of ${\cal O}(m \log N)$ \cite{HIKP12a} for $m$-sparse signals, or even ${\cal O}(m \log m)$, see e.g.\ \cite{LWC13, PR13}.
However, these algorithms  succeed to compute the correct result only with constant probability being smaller than $1$.
Moreover, there is no efficient method  to check the correctness of the result.  Another drawback  is that samples  need to be drawn  randomly, this is significantly more complicated to realize by hardware. 
Regarding the numerical results, runtimes start to be more efficient than standard FFT for $m=50$ and $N >2^{17}$, see \cite{GIIS14, Sch13} for most of the considered  algorithms, and for  $m=50$ and $N >2^{14}$ for the algorithm in \cite{HIKP12b}. Another   experiment shows for fixed $N=2^{22}$ that several sparse FFT algorithms start to pay off for $m <2200$ while the algorithm by \cite{HIKP12b} is more efficient for $m \le 2^{17}$, \cite{Sch13}. Newest algorithms are even faster, but the probability to provide the exact result decreases with  larger $m$.
For a nice overview of the techniques of randomized sparse Fourier transforms we refer to \cite{GIIS14}.
\smallskip

Deterministic sparse FFT algorithms proposed e.g. in \cite{Aka10,Aka14, I10, I13, LWC13} are applicable to (approximately) sparse  or so-called norm bounded signals. Iwen \cite{I10,I13} considered  sparse $m$-term approximations  of the discrete Fourier transform and achieved an 
$\ell^{2}, \, \ell^{1}/\sqrt{m}$ error bound  with ${\cal O}(m^{2} \log^{4} N)$ operations.
Similar runtime is achieved for the deterministic algorithm in \cite{LWC13}, where accessibility not only to the usual signal  samples but also to time-shifted samples of the input function is assumed. 
The deterministic algorithm in \cite{Aka14} for signals being norm bounded by $\sqrt{m}$, the evaluation of $\delta$-approximations of all $\tau$-significant frequencies, has polynomial costs in $\log N, \, m, \, 1/\tau$ and  $1/\delta$. 
\smallskip

Finally, we mention sparse FFT algorithms based on Prony's method that are completely deterministic and can be operated with ${\cal O}(m^{3})$ operations (independent of the signal size $N$), see e.g. \cite{HKPV13, PT14},  but usually with an unstable numerical behavior. A better numerical performance can be achieved by randomization leading to a complexity $m^{5/3} \log^{2} N$, see \cite{HKPV13}.
\medskip

Compared to the approaches above, we restrict ourselves  to signal vectors possessing a small support interval.
Vectors with small support appear in different applications as e.g. in X-ray microscopy, where compact support is a frequently  used a-priori condition  in phase retrieval, as well as in computer tomography reconstructions. 
\medskip

{\bf Notations.}
First, we fix some notations.
For a vector $\mathbf x\in\C^N$ of length $N$, let its discrete Fourier transform be defined by $\widehat{\mathbf x}={\mathbf F}_N \mathbf x$, where the Fourier matrix ${\mathbf F}_N\in\C^{N\times N}$ is given by
\[{\mathbf F}_N := \left( \omega_N^{jk} \right)_{j,k=0}^{N-1}, \qquad \omega_N := \e^{\frac{-2\pi\ii}{N}}.\]

Let the {\it support  length} $m = | {\rm supp} \, {\bf x}|$ of ${\bf x} \in \C^N$ be defined as the minimal integer $m$ for which there exists a 
 $\mu\in\{0,\dots,N-1\}$ such that the components $x_k$ of $\mathbf x$ vanish for all $k\notin I := \{(\mu+r)\mod N,\quad r=0,\dots,m-1\}$. The index set $I$ is called {\it support interval} of ${\bf x}$. Obviously, the support length $m$ of ${\bf x}$ is
 an upper bound for its sparsity $\| {\bf x} \|_{0}$, i.e., the number of nonzero components of ${\bf x}$, since there may be zero components of ${\bf x}$ inside the support interval $I$. However, we always have $x_{\mu} \neq 0$ and $x_{\mu+m-1} \neq 0$.
Observe that  the support length  $m$ of a vector $\mathbf x \in {\C}^{N}$ is always uniquely defined  
while the support interval  itself resp.\ the first support index  $\mu$ needs not to be unique. 
For example, considering the vector ${\bf x} = (x_{k})_{k=0}^{N-1} \in C^{N}$ of even length $N$ given by 
$x_{1} = x_{N/2+1} =1$ and $x_{k} =0$ for $k \in \{ 0, \ldots, N-1\} \setminus \{ 1, N/2 +1 \}$, we find for ${\bf x}$ the support length $m= N/2+1$ while the support interval can be chosen either as $\{ 1, \ldots , N/2+1 \} $ or as 
$\{ N/2+1, \ldots , N-1, 0, 1 \}$.
However, if $m \le \frac{N}{2}$, then the support interval and hence also the first support  index $\mu$ are uniquely determined. 
\medskip

{\bf Our contribution.}
In this paper, we will derive new deterministic algorithms to reconstruct a vector ${\bf x}\in\C^N$ with support length $m< N$
from its discrete Fourier transform $\widehat{\bf x} \in {\C}^{N}$.
In case of exact data, we will use less than $4m$ Fourier samples  and ${\cal O}(m \log_{2} m)$ arithmetical operations  to recover ${\bf x}$. Thus, the algorithm already starts to pay off for $m < \frac{N}{4}$.

In case of noisy Fourier measurements, we will
 use  ${\cal O}(m \log_{2} N)$ arithmetical operations. 
More exactly, we will employ one or several FFTs of size less than $4m$ 
as well as the computation of $\lfloor \log_{2} \frac{N}{m} \rfloor-1$ scalar products of length $m$ to recover the full vector in a stable way.

In both cases, for exact as well as for noisy measurements, the algorithms are based on the idea that for $N=2^{J}$ the nonzero components of ${\bf x}$ can already be computed by evaluating a periodization of ${\bf x}$ with length $2^{L}  \ge m$. Thus, for the complete reconstruction of ${\bf x}$,
we only need to compute in a second step the correct support interval resp.\ the first support index of ${\bf x}$.
\medskip

{\bf Organization of the paper.}
In Section 2, we recall some properties of the discrete Fourier transform that will be used in our approach.
Section 3 is devoted to the sparse FFT algorithm for $m$-sparse vectors in case of exact Fourier measurements. Section 4 considers a more stable variant of the different steps of the algorithm that make the method robust in presence of noise, and even improves the accuracy of the resulting vector in comparison to the usual FFT method while using only ${\cal O}(m \log N)$ arithmetical operations.
Finally, we present our numerical experiments in Section 5.

\section{Preliminaries}

%

Throughout the paper let $N:=2^J$ with some $J>0$. 

We consider the periodizations $\mathbf x^{(j)}\in\C^{2^j}$ of $\mathbf x$,
\begin{align}\label{periodization}
\mathbf x^{(j)} = (x_k^{(j)})_{k=0}^{2^j-1} = \left( \sum_{\ell=0}^{2^{J-j}-1} x_{k+2^j\ell} \right)_{k=0}^{2^j-1}
\end{align}
for $j=0,\dots,J$. Obviously, $\mathbf x^{(0)} = \sum_{k=0}^{N-1} x_k$ is the sum of all components of $\mathbf x$, 
$\mathbf x^{(1)} = ( \sum_{k=0}^{N/2-1} x_{2k}, \sum_{k=0}^{N/2-1} x_{2k+1} )^T$ and $\mathbf x^{(J)} = \mathbf x$. We recall the following relationship for the discrete Fourier transform of the vectors $\mathbf x^{(j)}$, $j=0,\dots,J$, in terms of $\widehat{\mathbf x}$.

\begin{lemma} \label{lemma1}
For the vectors ${\mathbf x}^{(j)} \in \C^{2^j}$, $j=0,\dots,J$, in {\rm (\ref{periodization})}, we have the discrete Fourier transform
\[\widehat{\mathbf x}^{(j)} := {\mathbf F}_{2^j} \mathbf x^{(j)} = (\widehat{x}_{2^{J-j}k})_{k=0}^{2^j-1},\]
where $\widehat{\mathbf x} = (\widehat{x}_k)_{k=0}^{N-1} = {\mathbf F}_N \mathbf x$ is the Fourier transform of $\mathbf x \in \C^N$.
\end{lemma}

\begin{proof}
By definition, we find for the components $\widehat{x}_k^{(j)}$ of $\widehat{\mathbf x}^{(j)}$
\begin{align*}
\widehat{x}_k^{(j)} :&= \sum_{r=0}^{2^j-1} x_r^{(j)}\, \omega_{2^j}^{rk} = \sum_{r=0}^{2^j-1} \sum_{\ell=0}^{2^{J-j}-1} x_{r+ 2^j \ell}\, \omega_N^{2^{J-j}rk}\\
	&= \sum_{n=0}^{N-1} x_n\, \omega_N^{2^{J-j} n k} = \widehat{x}_{2^{J-j}k}, \qquad k=0,\dots,2^j-1.
\end{align*}
Thus the assertion follows.
\end{proof}

Lemma \ref{lemma1} tells us that for a given vector $\widehat{\bf x}$  the Fourier transform of the periodized vectors $\widehat{\bf x}^{(j)}$ needs not to be computed  but is just obtained by picking suitable components of $\widehat{\bf x}$.
Further, we want to make use of the following simple observation.

\begin{lemma} \label{lemma2}
Let ${\bf x} = (x_{k})_{k=0}^{N-1}\in {\C}^{N}$, $N= 2^{J}$, and let 
${\bf y} = (y_{k})_{k=0}^{N-1} \in {\C}^{N}$ be its shifted version  with components $y_{k}:=x_{(k+2^{j} \nu) {\rm mod} N}$, $k=0, \ldots , N-1$, for some $j \in \{ 0, \ldots , J-1 \}$ and  $\nu \in \{ 0, \ldots , 2^{J-j}-1 \}$.
Then the components of the Fourier transformed vectors $\widehat{\bf x} = {\bf F}_{N} {\bf x}$, $\widehat{\bf y} = {\bf F}_{N} {\bf y}$ satisfy
$$ \widehat{y}_{l} = \omega_{2^{J-j}}^{-l\nu} \, \widehat{x}_{l}, \qquad l=0, \ldots , N-1. $$
In particular, for $j=J-1$ and $\nu =1$ we have $y_{k}= x_{(k+N/2) {\rm mod}  N}$, $k=0, \ldots , N-1$, with
$$ \widehat{x}_{2k} = \widehat{y}_{2k}, \qquad \widehat{x}_{2k+1} = - \widehat{y}_{2k+1}, \qquad  k=0, \ldots , \frac{N}{2} -1. $$
\end{lemma}

\begin{proof}
With $\widehat{\bf x} = {\bf F}_{N} {\bf x} = (\widehat{x}_{l})_{l=0}^{N-1}$ and  $\widehat{\bf y} = {\bf F}_{N} {\bf y} = (\widehat{y}_{l})_{l=0}^{N-1}$ we obtain
$$ \widehat{y}_{l} = \sum_{k=0}^{N-1} y_{k} \, \omega_{N}^{lk} = \sum_{k=0}^{N-1} x_{(k+2^{j} \nu) {\rm mod} N} \, \omega_{N}^{lk} =
\sum_{k=0}^{N-1} x_k \, \omega_{N}^{l(k-2^{j} \nu)} = \omega_{2^{J-j}}^{-l\nu} \, \widehat{x}_{l}.$$
For $j=J-1$ and $\nu =1$ the assertion follows with $\omega_{2^{J-j}}^{-l\nu} = \omega_{2}^{-l} = (-1)^{l}$.
\end{proof}

\section{Reconstruction of $\mathbf x$ from exact Fourier data}

We assume that the Fourier transformed vector  $\widehat{\bf x} = {\bf F}_{N} {\bf x}$ is given, and that 
the support length $m$ of ${\bf x}$ or an upper bound $m$ of it is known a-priori. Now, we apply the following simple procedure to recover $\mathbf x$.

Let $2^{L-1}< m \le 2^L$. Then, by Lemma 1, $\widehat{\mathbf x}^{(L+1)} = (\widehat{x}_{2^{J-(L+1)}k})_{k=0}^{2^{L+1}-1}$ is the Fourier transform of $\mathbf x^{(L+1)}$. In a first step, we  compute $\mathbf x^{(L+1)}$ using inverse FFT of length $2^{L+1}$.

Since $|\supp \mathbf x|=m\le 2^L$ by assumption, it follows that $\mathbf x^{(L+1)}$ has already the same support length as ${\bf x}$ since for each $k \in \{ 0, \ldots , 2^{L+1} -1 \}$ the sum in
\begin{equation}\label{1}
x_k^{(L+1)} = \sum_{\ell=0}^{2^{J-L-1}-1} x_{k+2^{L+1} \ell} 
\end{equation} 
contains at most one nonvanishing term. 
Moreover, also the support itself, or equivalently the first support index $\mu = \mu^{(L+1)}$ of $\mathbf x^{(L+1)}$, is uniquely determined.
\medskip

Thus, in order to recover ${\bf x}$ from ${\bf x}^{(L+1)}$, we only need to compute the correct first support index $\mu^{(J)}$ of ${\bf x}$, such that 
the components of ${\bf x}$ are determined by 
$$ x_{(\mu^{(J)} +k) {\rm mod} \, N} = \left\{ \begin{array}{ll} x^{(L+1)}_{(\mu^{(L+1)} +k) {\rm mod} \, 2^{L+1}} & k=0, \ldots , m-1, \\
0 & k=m, \ldots , N-1.
\end{array} \right.
$$
From (\ref{1}) we know that $\mu^{(J)}$ is of the form $\mu^{(J)} = \mu^{(L+1)} + 2^{L+1}\nu $ for some $\nu \in 
\{ 0, 1, \ldots , 2^{J-L-1}-1 \}$.

\begin{theorem}
Let $\mathbf x\in \C^N$, $N=2^{J}$, have support length $m$ (or a support length bounded by $m$) with $2^{L-1} <  m \le 2^L$.
For $L < J-1$, let ${\bf x}^{(L+1)}$ be the $2^{L+1}$-periodization  of ${\bf x}$.  Then ${\bf x}$ can be uniquely recovered from ${\bf x}^{(L+1)}$  and one nonzero component of the vector $(\widehat{x}_{2k+1})_{k=0}^{N/2-1}$.
\end{theorem}

\begin{proof}
Let ${\bf x}^{(L+1)}$  have the support interval $\{ (\mu^{(L+1)} +r) \mod 2^{L+1}, \, r=0, \ldots , m-1 \}$. Since $m \le 2^{L}$, this support interval is uniquely determined. Further, we know that by construction  the first index $\mu^{(J)}$ of the support interval of ${\bf x}$  has the form $\mu^{(J)} = \mu^{(L+1)} + 2^{L+1} \nu$ for some $\nu \in \{ 0, \ldots , 2^{J-L-1}-1 \}$.
We now consider the vector ${\bf u}^0 \in {\C}^{N}$ that is given by the components 
$$ u^0_{\mu^{(L+1)} + k} = \left\{ \begin{array}{ll}
x^{(L+1)}_{(\mu^{(L+1)} +k) {\rm mod}  2^{L+1}} & k=0, \ldots , m-1, \\
0 & k=m, \ldots , N-1. \end{array} \right.
$$
The vector ${\bf u}^0$ is obtained  as one possible vector in ${\C}^{N}$ with support length $m$ possessing  the $2^{L+1}$-periodization  ${\bf x}^{(L+1)}$.
Obviously the first index of the support interval of ${\bf u}^0$  is $\mu^{(L+1)}$. Therefore, all further vectors in ${\C^{N}}$ with a support length $m$ and possessing  the $2^{L+1}$-periodization  ${\bf x}^{(L+1)}$ can be written as  a shifted version of ${\bf u}^0$ of the form 
$${\bf u}^{\nu} := (u_{k}^{\nu})_{k=0}^{N-1} \qquad  \mbox{with} \qquad  u_{k}^{\nu} = u^0_{(k + 2^{L+1} \nu) {\rm mod} N}, \; k=0, \ldots, N-1,
$$
for some $\nu \in \{ 1, \ldots , 2^{J-L-1}-1 \}$. Thus,  the desired vector  ${\bf x}$ is contained in the set  $\{ {\bf u}^{\nu}: \, \nu=0, \ldots , 2^{J-L-1}-1 \}$. 
\smallskip

It remains to find $\nu$ such that ${\bf x} = {\bf u}^{\nu}$.
From Lemma \ref{lemma2} it follows for the components of the Fourier transform $\widehat{\bf u}^\nu = (\widehat{u}_{l}^{\nu})_{l=0}^{N-1}$
 that $\widehat{u}^{\nu}_{l} = \omega_{2^{J-L-1}}^{-l \nu} \, \widehat{u}^0_{l}$.
 
Choosing now an odd-indexed nonzero Fourier value $\widehat{x}_{2k_{0}+1}$ of the given vector $\widehat{\bf x}$, we can compare it to 
the corresponding component of $\widehat{\bf u}^{0}$,
$$ \widehat{u}^0_{2k_{0}+1} = \sum_{r=0}^{m-1} x_{(\mu^{(L+1)} +r){\rm mod}2^{L+1}}^{(L+1)} \, \omega_{N}^{(\mu^{(L+1)} + r)(2k_{0}+1)}
$$
and obtain the correct value $\nu$ from 
$\omega_{2^{J-L-1}}^{-(2k_0+1)\nu} = \widehat{x}_{2k_0+1}/ \widehat{u}^0_{2k_0+1}$. 
We remark that $\nu \in \{ 0, \ldots , 2^{J-L-1} -1 \}$ is indeed uniquely determined by $\omega_{2^{J-L-1}}^{-(2k_0+1)\nu}$ since gcd$(2k_0+1, 
2^{J-L-1}) =1$, i.e.,  $2k_0+1$ and $2^{J-L-1}$ are mutually prime.
Finally, the vector ${\bf x}$ is obtained as ${\bf x} = {\bf u}^{\nu}$. \end{proof}

We summarize the algorithm to evaluate ${\bf x}$ from $\widehat{\bf x}$ for exact Fourier data as follows.

\begin{algorithm}\label{alg1} \null (Sparse FFT for vectors with small support for exact Fourier data) \\
\textbf{Input:} $\widehat{\mathbf x}\in\C^N$, $N=2^J$, $|{\rm supp}\,  {\bf x}| \le m<N$.

\begin{itemize}
\item Compute $L$ such that $2^{L-1}< m \le 2^L$, i.e., $L:= \lceil \log_{2} m \rceil$.
\item If $L=J$ or $L=J-1$, compute $\mathbf x= {\mathbf F}_N^{-1} \widehat{\bf x}$ using an FFT algorithm of length $N$.
\item If $L<J-1$:
\begin{enumerate}
\item Choose $\widehat{\bf a} :=(\widehat{x}_{2^{J-(L+1)}k})_{k=0}^{2^{L+1}-1}$ and compute 
$$\mathbf x^{(L+1)}:={\mathbf F}_{2^{L+1}}^{-1} \widehat{\mathbf a}$$ 
using an FFT algorithm for the inverse discrete Fourier transform of length $2^{L+1}$.

\item Determine the first support index $\mu^{(L+1)}\in\{0,\dots,2^{L+1}-1\}$ of ${\bf x}^{(L+1)}$ such that $x_{\mu^{(L+1)}}^{(L+1)}\neq0$ and $x_k^{(L+1)}=0$ for $k\notin\{(\mu^{(L+1)}+r){\rm mod}\, 2^{L+1}, \,  r=0,\dots, m-1\}$.

\item
Choose a Fourier component $\widehat{x}_{2k_{0}+1} \neq0$ of $\widehat{\bf x}$ and compute the sum
$$\widehat{u}^0_{2k_{0}+1}:=\sum_{\ell=0}^{m-1} x^{(L+1)}_{(\mu^{(L+1)}+\ell) {\rm mod} \,  2^{L+1}} \, \omega_{N}^{(2k_{0}+1)(\mu^{(L+1)}+\ell)}. $$

\item Compute the quotient $b:=\widehat{x}_{2k_{0}+1}/\widehat{u}^0_{2k_{0}+1}$ that is by construction of the form $b= \omega_{2^{J-L-1}}^p$ for some $p \in \{
0, \ldots , 2^{J-L-1}-1 \}$, and find $\nu \in \{
0, \ldots , 2^{J-L-1}-1 \}$ such that $(2k_{0}+1) \, \nu = p\, {\rm mod}\, 2^{J-L-1}$.

\item Set $\mu^{(J)} := \mu^{(L+1)} + 2^{L+1} \nu$, 
 and $\mathbf x := (x_{k})_{k=0}^{N-1}$ with entries
\begin{align*}
x_{(\mu^{(J)}+\ell)\text{\rm mod}\, N} :=\left\{ \begin{array}{ll} x^{(L+1)}_{(\mu^{(L+1)}+\ell)\text{\rm mod}\,  2^{L+1}} & \ell=0,\dots,m-1,\\ 0 & \ell=m, \ldots , N-1.  \end{array} \right.
\end{align*}
\end{enumerate}
\end{itemize}
\textbf{Output:} $\mathbf x$.
\end{algorithm}

We simply observe that the proposed algorithm has an arithmetical complexity of ${\cal O}(m \log m)$.
In step 1 we employ an FFT algorithm of length $2^{L+1} < 4m$ having this complexity. All other steps
can be performed with ${\cal O} (m)$ or less operations. 
Particularly, the support interval (resp.\ the first support index $\mu^{(L+1)}$ 
of ${\bf x}^{(L+1)}$) in step 2 can e.g.\ be found by computing the local energies $e_{k} := \sum_{\ell=k}^{m+k-1} |x_{\ell \,{\rm mod} \, 2^{L+1}}^{(L+1)}|^{2}$ for $k=0, \ldots , 2^{L+1}-1$, and taking $\mu^{(L+1)} := \mbox{argmax}_{k} \, e_{k}$. Here, $e_{0}$ can be computed by at most ${\cal O}(m)$ operations, and 
all further energies by using the recursion $e_{k+1} = e_{k} - |x_{k}|^{2}+ |x_{k+m}|^{2}$, where we need to keep in mind that there are only
$m$ nonzero entries in ${\bf x}^{(L+1)}$. 

The complete algorithm  requires less then $4m$ Fourier values for the 
overall evaluation.

Let us give some further remarks on Algorithm \ref{alg1}.

\begin{remark} \hspace{1cm} 
\smallskip

It is always possible to choose a nonzero Fourier component of the vector \linebreak 
$ (\widehat{x}_{2k+1})_{k=0}^{N/2-1}$. This can be seen as follows.
Let ${\rm supp} \, {\bf x}= \{\mu^{(J)}, \mu^{(J)}+1 \,{\rm mod} \, N, \dots, \mu^{(J)}+m-1 \,{\rm mod} \, N\}$ be the support of ${\bf x}$, where $m \le 2^{L} \le N/4$. 
 Then the trigonometric polynomial
\begin{equation}\label{p1}
 p(\omega) = \left| \sum_{k=0}^{N-1} x_{k} \, \e^{-\ii\omega k} \right|^{2}
= \left| \e^{-\ii  \omega \mu^{(J)}} \sum_{\ell=0}^{m-1} x_{(\mu^{(J)} + \ell) \, {\rm mod} \, N} \e^{-\ii \omega \ell} \right|^{2} 
\end{equation}
is real, non-negative and of degree $\le m-1$. Hence it possesses at most $m-1$ pairwise different zeros, where 
all zeros are at least double zeros. Observing now that 
$|\widehat{x}_{k} |^{2} = p\left( \frac{2 \pi k}{N}\right)$, we can conclude that not
all $\widehat{x}_{2k+1}$, $k=0, \ldots , N/2-1$, can be zeros of $p(\omega)$ since $N/2 \ge 2m $.
\end{remark}

\begin{remark} \label{rem3.4}\hspace{1cm} 
\smallskip

As shown in the remark above, there can be at most $m-1$ zero components in the vector  $ (\widehat{x}_{2k+1})_{k=0}^{N/2-1}$.
However, for a stable computation of the correct first support index $\mu^{(J)}$, the Fourier value $\widehat{x}_{2k_{0}+1}$ used in Algorithm {\rm \ref{alg1}} should have large modulus.
This may be ensured by taking 
$$ k_{0} := {\rm argmax} \, \{ |\widehat{x}_{2k+1}|^{2}: \, k=0, \ldots , N/2-1 \}. $$
Unfortunately, this procedure is quite costly. Instead, we may compute  
$$ k_{0}^{(L+1)} := {\rm argmax} \, \{ |\widehat{x}_{2^{J-(L+1)}k}|^{2}: \, k=0, \ldots , 2^{L+1}-1 \}. $$
Then we have for the trigonometric polynomial in {\rm (\ref{p1})},
$$| \widehat{x}_{2^{J-(L+1)} k_{0}^{(L+1)}} |^{2} = \textstyle p \left(\frac{2\pi k_{0}^{(L+1)}}{2^{L+1}} \right) = \max\limits_{k=0, \ldots , 2^{L+1}-1} p \left(\frac{2\pi k}{2^{L+1}} \right) >0,$$
and it is likely that $p\left(\frac{2\pi k_{0}^{(L+1)}}{2^{L+1}}\right)$ is close to the global maximum $\| p \|_{\infty}$ of $p(\omega)$.
In step 3 of Algorithm {\rm \ref{alg1}} we may now choose the one neighboring  Fourier value with maximal modulus,
 i.e., either  $\widehat{x}_{2^{J-(L+1)} k_{0}^{(L+1)}+1}$ or $\widehat{x}_{2^{J-(L+1)} k_{0}^{(L+1)}-1}$.

If $p(\omega)$ attains its global maximum at some point $\omega_{0} \in \Big[\frac{2\pi (k_{0}^{(L+1)}-1/2)}{2^{L+1}}, \frac{2\pi (k_{0}^{(L+1)}+1/2)}{2^{L+1}}\Big)$, then it follows that $\Big|\omega_{0} - \frac{2\pi k_{0}^{(L+1)}}{2^{L+1}}\Big| \le \frac{\pi}{2^{L+1}}$.
With the above choice of $\widehat{x}_{2k_{0}+1}$ with $ k_{0}=2^{J-L-2} k_{0}^{(L+1)}$ or $k_{0}= 2^{J-L-2} k_{0}^{(L+1)} -1$,  we can further assume that 
$\left|\omega_{0} - \frac{2\pi (2k_{0}+1)}{N}\right| \le \frac{\pi}{2^{L+1}}$, and applying the Lemma of Ste\v{c}kin (see e.g. \cite{Gr99}), we find
$$ \textstyle |\widehat{x}_{2k_{0}+1}|^{2} = p \left(\frac{2\pi (2k_{0}+1)}{N} \right) \ge \| p \|_{\infty} \cos\left(\frac{\pi(m-1)}{2^{L+1}}\right) >0. $$
\end{remark}

\section{Reconstruction of $\mathbf x$ from noisy Fourier data}

Let us now assume that the given Fourier data are noisy, i.e., they are perturbed by uniform noise,  
\begin{equation}\label{e1} \widehat{y}_{k} = \widehat{x}_{k} + \epsilon_{k} 
\end{equation}
with $|\epsilon_{k}| \le \delta$. 
Similarly as before, we want to reconstruct ${\bf x}$ from the noisy Fourier vector $\widehat{\bf y}$ using the a-priori knowledge that 
${\bf x}$ has a support interval of length $m < N$.

For that purpose, we propose a stabilized variant of Algorithm \ref{alg1}.
The stabilization regards the following essential steps in the algorithm, namely \\
(1) the correct determination of the support interval of ${\bf x}^{(L+1)}$, \\
(2) the correct determination of the support interval resp.\ the first support index $\mu^{(J)}$ of the desired vector ${\bf x}$, and \\
(3) the evaluation of the nonzero components of ${\bf x}$  within the support interval that may be improved by employing more Fourier values $\widehat{\bf y}$ than in the case of exact data.
\medskip

\noindent
{\bf (1) Stable determination of the support interval of ${\bf x}^{(L+1)}$.}

For that purpose we use the following observation.

\begin{theorem}\label{theo41}
Let $\mathbf x\in \C^N$, $N=2^{J}$, have support length $m$ (or a support length bounded by $m$) with $2^{L-1} <  m \le 2^L$.
For $L < J-1$, let ${\bf x}^{(L+1)}= (x^{(L+1)}_{k})_{k=0}^{2^{L+1}-1}$ be the $2^{L+1}$-periodization  of ${\bf x}$ and $\widehat{\bf x}^{(L+1)}=(\widehat{x}_{2^{J-L-1}k})_{k=0}^{2^{L+1}-1}$ its Fourier transform.  
Then, for each shifted vector of the form
$$ \widehat{\bf z}^{(\kappa)} := \left( \widehat{x}_{2^{J-L-1}k + \kappa} \right)_{k=0}^{2^{L+1}-1}, \qquad \kappa = 0, \ldots , 2^{J-L-1} -1, $$
the inverse Fourier transform ${\bf z}^{(\kappa)} = (z_{\ell}^{(\kappa)})_{\ell=0}^{2^{L+1}-1} = {\mathbf F}_{2^{L+1}}^{-1} \widehat{\bf z}^{(\kappa)} $ satisfies
$$ |z_{\ell}^{(\kappa)}| = |x_{\ell}^{(L+1)}| \quad \ell =0, \ldots , 2^{L+1}-1. $$
\end{theorem}

\begin{proof}
By definition, we obtain for the components $z_{\ell}^{(\kappa)}$ of ${\bf z}^{(\kappa)}$
\begin{eqnarray*}
z_{\ell}^{(\kappa)}  &=& \frac{1}{2^{L+1}} \sum_{k=0}^{2^{L+1}-1} \widehat{x}_{2^{J-L-1}k+\kappa} \, \omega_{2^{L+1}}^{-k\ell} \\
&=& \frac{1}{2^{L+1}} \sum_{k=0}^{2^{L+1}-1} \sum_{r=0}^{N-1} x_{r} \, \omega_{N}^{r(2^{J-L-1}k+\kappa)}\, \omega_{2^{L+1}}^{-k\ell} \\
&=&  \frac{1}{2^{L+1}} \sum_{r=0}^{N-1} x_{r} \, \omega_{N}^{r\kappa} \sum_{k=0}^{2^{L+1}-1} \omega_{2^{L+1}}^{k(r-\ell)} \\
&=& \sum_{j=0}^{2^{J-L-1}-1} x_{\ell + 2^{L+1}j} \, \omega_{N}^{(\ell + 2^{L+1}j) \kappa} 
= \omega_{N}^{\ell \kappa} \, \sum_{j=0}^{2^{J-L-1}-1} x_{\ell + 2^{L+1}j} \, \omega_{N}^{2^{L+1}j \kappa}.
\end{eqnarray*}
Since $| {\rm supp} \, {\bf x}| = m < 2^{L+1}$, for each $\ell$ the above sum contains only one value, thus
$|z_{\ell}^{(\kappa)}| = |x_{\ell}^{(L+1)} |$. 
\end{proof}
\medskip

Obviously, we have ${\bf z}^{(0)} = {\bf x}^{(L+1)}$ by definition. We observe that all vectors ${\bf z}^{(\kappa)}$ are constructed from different Fourier components of $\widehat{\bf x}$. This circumstance will be used to stabilize the algorithm.
Applying the noisy measurements $\widehat{y}_{k}$ in (\ref{e1}) instead of $\widehat{x}_{k}$, let $\widehat{\widetilde{\bf z}}^{(\kappa)} 
:= (\widehat{y}_{2^{J-L-1}k+\kappa})_{k=0}^{2^{L+1}-1}$.
As before in Algorithm \ref{alg1}, we compute in a first step the vector $\widetilde{\bf z}^{(0)}= {\bf y}^{(L+1)}$ from the noisy measurements $(\widehat{y}_{2^{J-(L+1)}k})_{k=0}^{2^{L+1}-1}$. 
In order to determine the support interval of ${\bf x}^{(L+1)}$ from the vector $\widetilde{\bf z}^{(0)}$
we consider the energies 
$$ \widetilde{e}_{k}^{(0)} := \sum_{\ell = k}^{m+k-1} |y_{\ell \,{\rm mod} 2^{L+1}}^{(L+1)} |^{2} = \sum_{\ell = k}^{m+k-1} |\widetilde{z}_{\ell \,{\rm mod} \, 2^{L+1}}^{(0)} |^{2}, \qquad k=0, \ldots , 2^{L+1} -1, $$
and take $\mu_{0}^{(L+1)} := \mathop{\rm argmax}\limits_{k} \widetilde{e}_{k}^{(0)}$ as the first estimate for $\mu^{(L+1)}$.
For higher noise levels, we stabilize the computation of $\mu^{(L+1)}$ as follows:
We compute also the vector $\widetilde{\bf z}^{(2^{J-L-2})}$ by applying the inverse FFT to $\widehat{\widetilde{\bf z}}^{(2^{J-L-2})} = ( \widehat{y}_{2^{J-L-2}(2k+1)})_{k=0}^{2^{L+1}-1}$, compute the energies 
$$ \widetilde{e}_{k}^{(1)} : = \sum_{\ell = k}^{m+k-1} |\widetilde{z}_{\ell \, {\rm mod} \,  2^{L+1}}^{(2^{J-L-2})} |^{2}, \qquad k=0, \ldots , 2^{L+1} -1, 
$$
and take $\mu_{1}^{(L+1)} := \mathop{\rm argmax}\limits_{k} \frac{1}{2} ( \widetilde{e}_{k}^{(0)} + \widetilde{e}_{k}^{(1)})$.
If $\mu_{1}^{(L+1)} = \mu_{0}^{(L+1)}$, then this index is taken as the first support index  of ${\bf x}^{(L+1)}$.
 Otherwise, we compute a further vector, e.g.,  $\widetilde{\bf z}^{(2^{J-L-3})} = {\bf F}_{2^{L+1}}^{-1} (\widehat{y}_{2^{J-L-3}(4k+1)})_{k=0}^{2^{L+1}-1}$, evaluate the corresponding energies $\widetilde{e}_{k}^{(2)}$ and
 $\mu_{2}^{(L+1)} :=  \mathop{\rm argmax}\limits_{k} \frac{1}{3} ( \widetilde{e}_{k}^{(0)} + \widetilde{e}_{k}^{(1)}+ \widetilde{e}_{k}^{(2)})$, etc.

\medskip

\noindent
{\bf (2) Stable determination of the support interval of ${\bf x}$.}

In order to stabilize the computation  of the first  support index  $\mu^{(J)}$ of the complete vector ${\bf x} = {\bf x}^{(J)}$
resp.\  the shift $\nu$ such that $\mu^{(J)} = \mu^{(L+1)} - 2^{L+1} \nu$, we employ an iterative procedure, where
at each iteration step, the support interval of the periodized vector ${\bf x}^{(j+1)}$, $j=L+1, \ldots J-1$, is computed from the support interval of the  vector  ${\bf x}^{(j)}$ of half length.

At each iteration step we apply the following procedure.
Observing that ${\bf x}^{(j+1)}$ has the same support length $m$ as ${\bf x}^{(j)}$, and using the  relation
\[x_k^{(j+1)}+x_{k+2^{j}}^{(j+1)} = x_k^{(j)}, \qquad k=0,\dots, 2^{j}-1,\]
 it is sufficient to check whether $\mu^{(j+1)} = \mu^{(j)}$ or 
$\mu^{(j+1)} = \mu^{(j)}+ 2^{j}$, i.e., whether
the first support index $\mu^{(j)}$ of ${\bf x}^{(j)}$ is equal to the first support index $\mu^{(j+1)}$ 
of ${\bf x}^{(j+1)}$, or if the support has to be shifted by $2^{j}$.
We illustrate the two cases in detail in Figure \ref{support}.
\medskip

\begin{figure}[h]
{\small
\null \hspace{10mm} \textbf{First case:} $\mu^{(j+1)} = \mu^{(j)}$.
\begin{center}
\begin{minipage}{\linewidth}
\centering
\begin{tikzpicture}
\draw 	(0,0) -- (3,0);
\foreach \x in {0,3}
\draw[xshift=\x cm] (0pt,-2pt) -- (0pt,2pt);
\draw 	(0,0.1) node[left=5pt]{$\mathbf x^{(j)} $};
\draw[line width=2.5pt] (1,0) -- (2,0);
\end{tikzpicture}
\hspace{0.4cm}
$\Longrightarrow$
\hspace{0.4cm}
\begin{tikzpicture}
\draw 	(0,0) -- (6,0);
\foreach \x in {0,3,6}
\draw[xshift=\x cm] (0pt,-2pt) -- (0pt,2pt);
\draw 	(0,0.1) node[left=5pt]{$\mathbf x^{(j+1)} $};
\draw[line width=2.5pt] (1,0) -- (2,0);
\end{tikzpicture}
\end{minipage}
\end{center}
\begin{center}
\begin{minipage}{\linewidth}
\centering
\begin{tikzpicture}
\draw 	(0,0) -- (3,0);
\foreach \x in {0,3}
\draw[xshift=\x cm] (0pt,-2pt) -- (0pt,2pt);
\draw 	(0,0.1) node[left=5pt]{$\mathbf x^{(j)} $};
\draw[line width=2.5pt] 	(0,0) -- (0.5,0)
					(2.5,0) -- (3,0);
\end{tikzpicture}
\hspace{0.4cm}
$\Longrightarrow$
\hspace{0.4cm}
\begin{tikzpicture}
\draw 	(0,0) -- (6,0);
\foreach \x in {0,3,6}
\draw[xshift=\x cm] (0pt,-2pt) -- (0pt,2pt);
\draw 	(0,0.1) node[left=5pt]{$\mathbf x^{(j+1)} $};
\draw[line width=2.5pt] (2.5,0) -- (3.5,0);
\end{tikzpicture}
\end{minipage}
\end{center}

\null \hspace{10mm}\textbf{Second case:} $\mu^{(j+1)} = \mu^{(j)}+2^{j}$.
\begin{center}
\begin{minipage}{\linewidth}
\centering
\begin{tikzpicture}
\draw 	(0,0) -- (3,0);
\foreach \x in {0,3}
\draw[xshift=\x cm] (0pt,-2pt) -- (0pt,2pt);
\draw 	(0,0.1) node[left=5pt]{$\mathbf x^{(j)}$};
\draw[line width=2.5pt] (1,0) -- (2,0);
\end{tikzpicture}
\hspace{0.4cm}
$\Longrightarrow$
\hspace{0.4cm}
\begin{tikzpicture}
\draw 	(0,0) -- (6,0);
\foreach \x in {0,3,6}
\draw[xshift=\x cm] (0pt,-2pt) -- (0pt,2pt);
\draw 	(0,0.1) node[left=5pt]{$\mathbf x^{(j+1)}$};
\draw[line width=2.5pt] (4,0) -- (5,0);
\end{tikzpicture}
\end{minipage}
\end{center}
\begin{center}
\begin{minipage}{\linewidth}
\centering
\begin{tikzpicture}
\draw 	(0,0) -- (3,0);
\foreach \x in {0,3}
\draw[xshift=\x cm] (0pt,-2pt) -- (0pt,2pt);
\draw 	(0,0.1) node[left=5pt]{$\mathbf x^{(j)}$};
\draw[line width=2.5pt] 	(0,0) -- (0.5,0)
					(2.5,0) -- (3,0);
\end{tikzpicture}
\hspace{0.4cm}
$\Longrightarrow$
\hspace{0.4cm}
\begin{tikzpicture}
\draw 	(0,0) -- (6,0);
\foreach \x in {0,3,6}
\draw[xshift=\x cm] (0pt,-2pt) -- (0pt,2pt);
\draw 	(0,0.1) node[left=5pt]{$\mathbf x^{(j+1)}$};
\draw[line width=2.5pt] 	(0,0) -- (0.5,0)
					(5.5,0) -- (6,0);
\end{tikzpicture}
\end{minipage}
\end{center}
}
\caption{Possible support change in one iteration step.}
\label{support}
\end{figure}

Generally, in order to recover $\mathbf x^{(j+1)}$ from $\mathbf x^{(j)}$, we apply the following procedure.

\begin{theorem}
Let $\mathbf x\in \C^N$, $N=2^J$, have support of length $m$ (or a support length bounded by $m$) with $2^{L-1}<m\leq 2^L$.
Further, let  $\mathbf x^{(j)}$, $L+1\leq j\leq J-1$, be the periodizations of ${\bf x} = {\bf x}^{(J)}$ as given in {\rm (\ref{periodization})}. Then, for each $j=L+1, \ldots , J-1$, the vector $\mathbf x^{(j+1)}$ can be uniquely recovered from $\mathbf x^{(j)}$ and one nonzero component of the vector of Fourier values $\mathbf y := (\widehat{x}_{2^{J-(j+1)}(2k+1)})_{k=0}^{2^j-1}$.
\end{theorem}

\begin{proof}
Let ${\bf x}^{(j)}$ be the given vector of length $2^{j}$ with support 
$\supp {\bf x}^{(j)}=\{(\mu^{(j)}+r) \, {\rm mod}\,  2^j, \quad r=0,\dots,n-1\}$.
In order to determine ${\bf x}^{(j+1)}$, we only need to decide whether 
$\mu^{(j+1)} = \mu^{(j)}$ or $\mu^{(j+1)} = \mu^{(j)}+2^{j}$, or equivalently,
whether ${\bf x}^{(j+1)}$ is given by

\begin{align}\label{noshift}
x_{\mu^{(j)}+\ell}^{(j+1)}=\left\{ \begin{array}{ll} x_{(\mu^{(j)}+\ell)\, \text{mod} \, 2^j} & \ell=0,\dots,n-1,\\ 0 & \text{else,}  \end{array} \right.
\end{align}
or by
\begin{align}\label{shift}
x_{(\mu^{(j)}+2^j+\ell)\, \text{mod} \, 2^{j+1}}^{(j+1)}=\left\{ \begin{array}{ll} x_{(\mu^{(j)}+\ell)\, \text{mod} \, 2^j} & \ell=0,\dots,n-1,\\ 0 & \text{else.}  \end{array} \right.
\end{align}

The two possible solutions ${\bf x}^{(j+1)}$ only differ by a shift of all components by $2^{j}$. For simplicity let us denote the two solutions by ${\bf u}^{(0)}$ and ${\bf u}^{(1)}$. Then, according to Lemma \ref{lemma2}, the Fourier transformed vectors $\widehat{\bf u}^{(0)}= (\widehat{u}^{(0)}_{k})_{k=0}^{2^{j+1}-1}$ and $\widehat{\bf u}^{(1)}= (\widehat{u}^{(1)}_{k})_{k=0}^{2^{j+1}-1}$ satisfy
the relationship 
$$ \widehat{u}^{(0)}_{2k+1} = - \widehat{u}^{(1)}_{2k+1}, \qquad k=0, \ldots , 2^{j}-1.$$

Using now one nonzero  Fourier value $\widehat{x}_{2^{J-(j+1)}(2k_{0}+1)} = \widehat{x}_{2k_{0}+1}^{(j+1)}$, we can determine the correct vector ${\bf x}^{(j+1)}$.
For that purpose, we just compute 
$$ \widehat{u}^{(0)}_{2k_{0}+1} = \sum_{\ell=0}^{m-1} x^{(j)}_{(\mu^{(j)}+\ell)\, \text{mod} \, 2^j} \omega_{2^{j+1}}^{(2k_{0}+1)(\ell+\mu^{(j)})}
$$
and compare it to the Fourier value $\widehat{x}_{2^{J-(j+1)}(2k_{0}+1)}$.
If $|\widehat{u}^{(0)}_{2k_{0}+1} - \widehat{x}_{2^{J-(j+1)}(2k_{0}+1)}| < |\widehat{u}^{(0)}_{2k_{0}+1} + \widehat{x}_{2^{J-(j+1)}(2k_{0}+1)}|$, then 
${\bf x}^{(j+1)} = {\bf u}^{(0)}$ and $\mu^{(j+1)} = \mu^{(j)}$, otherwise, we find  ${\bf x}^{(j+1)} = {\bf u}^{(1)}$ and $\mu^{(j+1)} = \mu^{(j)}+ 2^{j}$.
\end{proof}
\medskip

\noindent
{\bf (3) Evaluation of the nonzero components of ${\bf x}$.}

Finally, we can use  the vectors $\widetilde{\bf z}^{(\kappa)} \in {\C}^{2^{L+1}}$  computed in step (1) also for a more exact evaluation of the  nonzero components
$x_{k}$ of ${\bf x}$ as follows.
First, we employ our investigations in Theorem \ref{theo41} and observe that
$$ \widetilde{z}_{\ell \, {\rm mod\, 2^{L+1}}}^{(\kappa)} = \omega_{N}^{\ell \kappa} \, (x_{(\ell+ 2^{L+1} \nu) \,  {\rm mod} \, N} \, \omega_{N}^{2^{L+1}\nu \kappa}), \qquad \ell = \mu^{(L+1)}, \ldots , \mu^{(L+1)} +m-1. $$
Using the equation $\mu^{(J)}=\mu^{(L+1)}+2^{L+1}\nu$, we can reformulate this as
$$ \widetilde{z}_{(\mu^{(L+1)}+k) \, {\rm mod\, 2^{L+1}}}^{(\kappa)} = x_{(\mu^{(J)}+k)\,{\rm mod}\, N} \, \omega_{N}^{\kappa (\mu^{(J)}+k)}, \qquad k = 0, \ldots , m-1. $$
Therefore, if  we have to compute the support entries $x_{(\mu^{(J)}+k)\,{\rm mod}\, N}$, $k = 0, \ldots , m-1$,
from the noisy measurements $(\widehat{y}_{k})_{k=0}^{N-1}$ of $\widehat{\bf x}$, we can take the average
$$ x_{(\mu^{(J)}+k)\,{\rm mod}\, N} = \frac {1}{B+1} \sum_{r=0}^{B} \widetilde{z}_{(\mu^{(L+1)}+k) \, {\rm mod\, 2^{L+1}}}^{(\kappa_{r})} \, \omega_{N}^{-\kappa_{r}(\mu^{(J)}+k)}, \qquad k = 0, \ldots , m-1, $$
where $B+1$ is the number of the vectors $\widetilde{\bf z}^{(\kappa)} = {\bf F}_{2^{L+1}}^{-1} (\widehat{y}_{2^{J-L-1}k+ \kappa})_{k=0}^{2^{L+1}-1}$ that we want to involve.

The complete algorithm to compute  $\mathbf x$ from its Fourier transform by a fast sparse FFT algorithm  can now be summarized as follows.

\begin{algorithm} \label{alg2} \null (Sparse FFT for vectors with small support for noisy Fourier data) \\
\textbf{Input:} noisy measurement vector $\widehat{\mathbf y}\in\C^N$, $N=2^J$, $|{\rm supp}\,  {\bf x}| \le m<N$.
\begin{itemize}
\item Compute $L$ such that $2^{L-1}< m \le 2^L$, i.e., $L:= \lceil \log_{2} m \rceil$.
\item If $L=J$ or $L=J-1$, compute $\mathbf x= {\mathbf F}_N^{-1} \widehat{\bf y}$ by inverse FFT.
\item If $L<J-1$:
\begin{enumerate}
\item Choose $\widehat{\widetilde{\bf z}}^{(0)} := \widehat{\bf y}^{(L+1)} =(\widehat{y}_{2^{J-(L+1)}k})_{k=0}^{2^{L+1}-1}$ and compute 
$\widetilde{\bf z}^{(0)} = \mathbf y^{(L+1)}:= {\mathbf F}_{2^{L+1}}^{-1} \widehat{\mathbf y}^{(L+1)}$
using an FFT algorithm for the inverse discrete Fourier transform.
\item Determine the first support index $\mu^{(L+1)}\in\{0,\dots,2^{L+1}-1\}$ of ${\bf x}^{(L+1)}$ 
using the following iteration:

\begin{itemize}
\item[$\bullet$] Compute the energies 
$$ \widetilde{e}_{k}^{(0)} :=  \sum_{\ell = k}^{m+k-1} |\widetilde{z}_{\ell \, {\rm mod} \, 2^{L+1}}^{(0)} |^{2}, \qquad k=0, \ldots , 2^{L+1} -1, $$
and compute $\mu_{0}^{(L+1)} := \mathop{\rm argmax}\limits_{k} \widetilde{e}_{k}^{(0)}$.

\item[$\bullet$] Compute $\widetilde{\bf z}^{(J-L-2)}$ by IFFT from $( \widehat{y}_{2^{J-L-2}(2k+1)})_{k=0}^{2^{L+1}-1}$, determine
$$ \widetilde{e}_{k}^{(1)} : = \sum_{\ell = k}^{m+k-1} |\widetilde{z}_{\ell \, {\rm mod} \,  2^{L+1}}^{(2^{J-L-2})} |^{2}, \qquad k=0, \ldots , 2^{L+1} -1, 
$$
and take $\mu_{1}^{(L+1)} := \mathop{\rm argmax}\limits_{k} \frac{1}{2} ( \widetilde{e}_{k}^{(0)} + \widetilde{e}_{k}^{(1)})$.

\item[$\bullet$] Set $j:=0$.\\
While $ \mu_{j}^{(L+1)} \neq \mu_{j+1}^{(L+1)}$\\
proceed by computing for a further $\kappa \in \{ 1, \ldots , 2^{J-L-1}-1 \} \setminus \{ 2^{J-L-2} \}$
the vector $\widetilde{\bf z}^{(\kappa)}$ by IFFT from $( \widehat{y}_{2^{J-L-1}k+ \kappa)})_{k=0}^{2^{L+1}-1}$, the energies 
$$ \widetilde{e}_{k}^{(j+2)} : = \sum_{\ell = k}^{m+k-1} |\widetilde{z}_{\ell \, {\rm mod} \,  2^{L+1}}^{(\kappa)} |^{2}, \qquad k=0, \ldots , 2^{L+1} -1, 
$$
and take $\mu_{j+2}^{(L+1)} := \mathop{\rm argmax}\limits_{k} \frac{1}{j+3} \sum_{r=0}^{j+2} \widetilde{e}_{k}^{(r)} $.\\
Set $\mu^{(L+1)} := \mu_{j+2}^{(L+1)}$ and $j:=j+1$.\\
End (while).\\
Set ${\bf x}^{(L+1)} = (x_{k}^{(L+1)})_{k=0}^{2^{L+1}-1}$ with 
$$x_{(k + \mu^{(L+1)})\, {\rm mod} \, 2^{L+1}}^{(L+1)}:=  \left\{ \begin{array}{ll}
\widetilde{z}^{(0)}_{(k + \mu^{(L+1)})\, {\rm mod} \, 2^{L+1}} & k= 0, \ldots ,  m-1, \\
0   & k= m, \ldots ,  2^{L+1}-1. \end{array} \right.
$$
\end{itemize}

\item For $j=L+1,\dots,J-1$

Choose a Fourier component $\widehat{y}_{2^{J-(j+1)}(2k_{0}+1)} = \widehat{y}_{2k_{0}+1}^{(j+1)}\neq0$ and compute 
$$a_{j+1}: = \sum_{\ell=0}^{m-1} x^{(L+1)}_{(\mu^{(L+1)}+\ell) \, {\rm mod} \,  2^{L+1}} \, \omega_{2^{j+1}}^{(2k_{0}+1)(\mu^{(j)}+\ell)}. $$

If $|a_{j+1}-\widehat{y}_{2k_{0}+1}^{(j+1)}| < |a_{j+1}+\widehat{y}_{2k_{0}+1}^{(j+1)}|$, then
set $\mu^{(j+1)} := \mu^{(j)}$ and $\mathbf x^{(j+1)} := (x^{(j+1)}_{k})_{k=0}^{2^{j+1}-1}$ with entries
\begin{align*}
x_{\mu^{(j)}+\ell}^{(j+1)}=\left\{ \begin{array}{ll} x^{(j)}_{(\mu^{(j)}+\ell)\, \text{\rm mod} \,2^j} & \ell=0,\dots,n-1,\\ 0 & {\rm else.}  \end{array} \right.
\end{align*}

If $|a_{j+1}-\widehat{y}_{2k_{0}+1}^{(j+1)}| \ge |a_{j+1}+\widehat{y}_{2k_{0}+1}^{(j+1)}|$, then
set  $\mu^{(j+1)} := \mu^{(j)} + 2^{j}$ and $\mathbf x^{(j+1)} := (x^{(j+1)}_{k})_{k=0}^{2^{j+1}-1}$ with entries
\begin{align*}
x_{(\mu^{(j)}+2^j+\ell)\text{\rm mod}\, 2^{j+1}}^{(j+1)}=\left\{ \begin{array}{ll} x^{(j)}_{(\mu^{(j)}+\ell)\text{\rm mod}\,  2^j} & \ell=0,\dots,n-1,\\ 0 & {\rm else.}  \end{array} \right.
\end{align*}
\item Assuming that $\widetilde{\bf z}^{(\kappa_{r})} \in {\C}^{2^{L+1}}$ , $r=0, \ldots , B$, have been  evaluated already in step 2, we 
compute 
$$ x_{(\ell + 2^{L+1} \nu) \, {\rm mod} \, N} = \frac {1}{B+1} \sum_{r=0}^{B} \widetilde{z}_{\ell}^{(\kappa_{r})} \, \omega_{N}^{-\kappa_{r}(\ell + 2^{L+1} \nu)},$$
$\ell = \mu^{(L+1)}, \ldots,  \mu^{(L+1)} +m-1$.
\end{enumerate}
\end{itemize}
\textbf{Output:} $\mathbf x^{(J)}=\mathbf x$.
\end{algorithm}

Let us shortly summarize the arithmetical complexity  of the algorithm.
Step 1 and 2 together require two or more inverse  FFTs of length $2^{L+1}< 4n$, i.e. ${\cal O} (m \log m) $ arithmetical operations. 
The number of involved vectors $\widetilde{\bf z}^{(\kappa)}$ depends on the noise level.
The evaluation of energies can be done at each iteration step by ${\cal O}(m)$ operations.

In step 3, $J-(L+1) = \log_{2} N- \lceil \log_{2} m \rceil -1 < \log_{2} (N/m) $ iterations are needed, where at each iteration  a scalar product  of length $m$  has to be computed and to be compared to a given Fourier value.
To find a nonzero Fourier value needs less than $m$ comparisons, see also Remark \ref{rem3.4}.
Hence only ${\cal O} (m \log m + m \log (N/m)) =  {\cal O}(m \log_{2} N)$ arithmetical operations  are needed to recover ${\bf x}$ in the case of noisy data.

\begin{remark} \hspace{1cm} 
\smallskip

For the computation of ${\bf x}$ it is sufficient to compute ${\bf x}^{(L+1)}$ as well as the first support index $\mu^{(J)}$, where $\mu^{(j+1)}$, $j=L+1, \ldots, J-1$, is iteratively computed from $\mu^{(j)}$. The values $a_{j+1}$ can be obtained directly from ${\bf x}^{(L+1)}$ and $\mu^{(j)}$.
Hence, there is no reason to compute the intermediate vectors ${\bf x}^{(j+1)}$ in step 3 of Algorithm \ref{alg2}, they are only given for better illustration.
\end{remark}

\section{Numerical results}

In this section, we discuss the behavior of the algorithm for noisy input data. We reconstruct randomly generated vectors $\bf x$ from disturbed Fourier data $\widehat{\bf y} = \widehat{\bf x}+\eps$, where we assume uniform noise $\eps = (\varepsilon_k)_{k=0}^{N-1}$ with $|\varepsilon_k| \le \delta$ at different noise levels. As a noise measure, we use the SNR value
$$\text{SNR} = 20\cdot \log_{10} \frac{\|\widehat{\bf x}\|_2}{\|\eps \|_2}.$$
Let us first illustrate the algorithm for a vector $\bf x$ of length $N=2^8=256$ and support length $m=6$, i.e., we have $J=8$ and $L=3$. The nonzero entries of $\bf x$ are $x_{105}=8$, $x_{107}=-3$, $x_{108}=-5$ and $x_{110}=2$. We add a noise vector $\eps$ of $\text{SNR}=20$ to $\widehat{\bf x}$ and reconstruct $\bf x$ from $\widehat{\bf y} = \widehat{\bf x} + \eps$. For the noise vector $\eps$ in this example, it holds that $\|\eps\|_\infty = 1.721$ and $\|\eps \|_1/256 = 0.952$. In Figure \ref{example} we present $\bf x$, and we compare the reconstruction results of our algorithm to the result of the inverse Fourier transform applied to $\widehat{\bf y}$. The reconstruction $\bf x'$ computed by our deterministic sparse FFT algorithm has nonzero components 
$x'_{105}=7.884 + 0.131\ii\,$, 
$x'_{106}=-0.070 - 0.149\ii\,$,
$x'_{107}=-3.100 + 0.171\ii\,$,
$x'_{108}=-4.955 + 0.094\ii\,$,
$x'_{109}=-0.105 + 0.022\ii\,$,
$x'_{110}=1.879 - 0.076\ii\,$,
 and an error $\|{\bf x}-{\bf x'}\|_2/256=0.00146$ whereas the error by the inverse FFT is $\|{\bf x}-{\bf F}_{256}^{-1}{\widehat{\bf y}}\|_2/256=0.00395$. In this example, the reconstruction by our algorithm requires no additional vectors for the determination of $\mu^{(L+1)}$ in step 2 of Algorithm \ref{alg2}, i.e., we only use the two vectors $\widetilde{\bf z}^{(0)}$  and $\widetilde{\bf z}^{(J-L-2)} = \widetilde{\bf z}^{(3)}$ of length $2^{L+1}=16$ in this step. Thus, we have taken $32+4=36 $ Fourier values to recover ${\bf x}$.

\begin{figure}[h!]
\begin{center}
\includegraphics[width=0.85\linewidth]{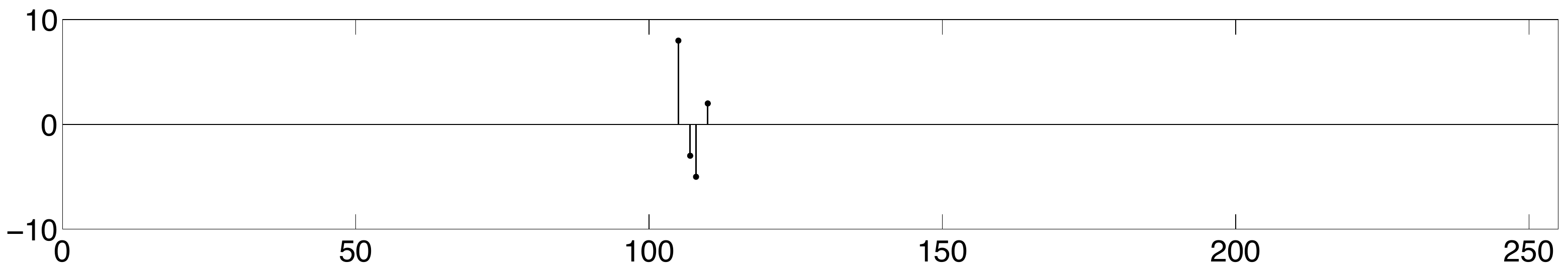}\\
{\small (a)} \\[1ex]
\includegraphics[width=0.85\linewidth]{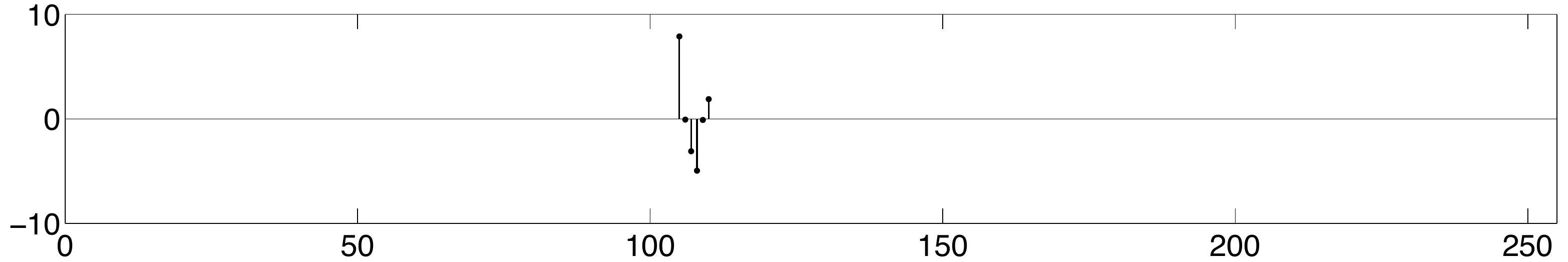}\\
{\small (b)}\\[1ex]
\includegraphics[width=0.85\linewidth]{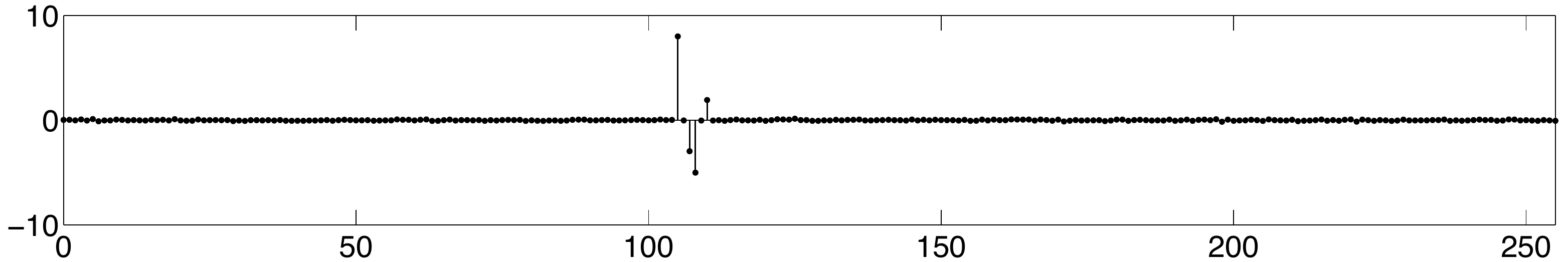}\\
{\small (c)}\\[1ex]
\caption{(a) Original vector ${\bf x}$ of length $N=256$; (b) Reconstruction of ${\bf x}$ using the sparse \\ FFT Algorithm \ref{alg2};
(c) Reconstruction of $\bf x$ using the inverse FFT.}
\label{example}
\end{center}
\end{figure}

In Figure \ref{algvsifft} we show the reconstruction error using the sparse FFT Algorithm \ref{alg2} for vectors $\bf x$ of length $N=2^{22}$ and support length $m=50$ where the vectors are randomly generated with $|\text{Re}(x_k)|\le 10$, $|\text{Im}(x_k)|\le 10$
for $k$ in the support interval. 
We consider noisy Fourier data with SNR values between 0 and 50 and compute the reconstruction $\bf x'$ for 100 vectors $\bf x$ at different noise levels. The quality of the reconstructed vectors $\bf x'$ is evaluated by computing the norm $\|{\bf x} - {\bf x'}\|_2/N$. Figure  \ref{algvsifft} shows the average error norm over all 100 considered vectors. We compare the reconstruction results of our algorithm to the results of an inverse Fourier transform applied to the noisy vectors $\widehat{\bf y}$. 
For  noise levels $\text{SNR} \le 10$, our algorithm made use of additional vectors $\widetilde{\bf z}^{(\kappa)}$ in step 2 of Algorithm \ref{alg2} in order to improve the identification of the first support index $\mu^{(L+1)}$ of ${\bf x}^{(L+1)}$. Here, the algorithm evaluated at most five additional vectors (for $\text{SNR}=0$), hence a total number of up to seven vectors $\widetilde{\bf z}^{(\kappa)}$ has been used to determine $\mu^{(L+1)}$ in some cases.

\begin{figure}[h!]
\begin{center}
\includegraphics[width=0.6\linewidth]{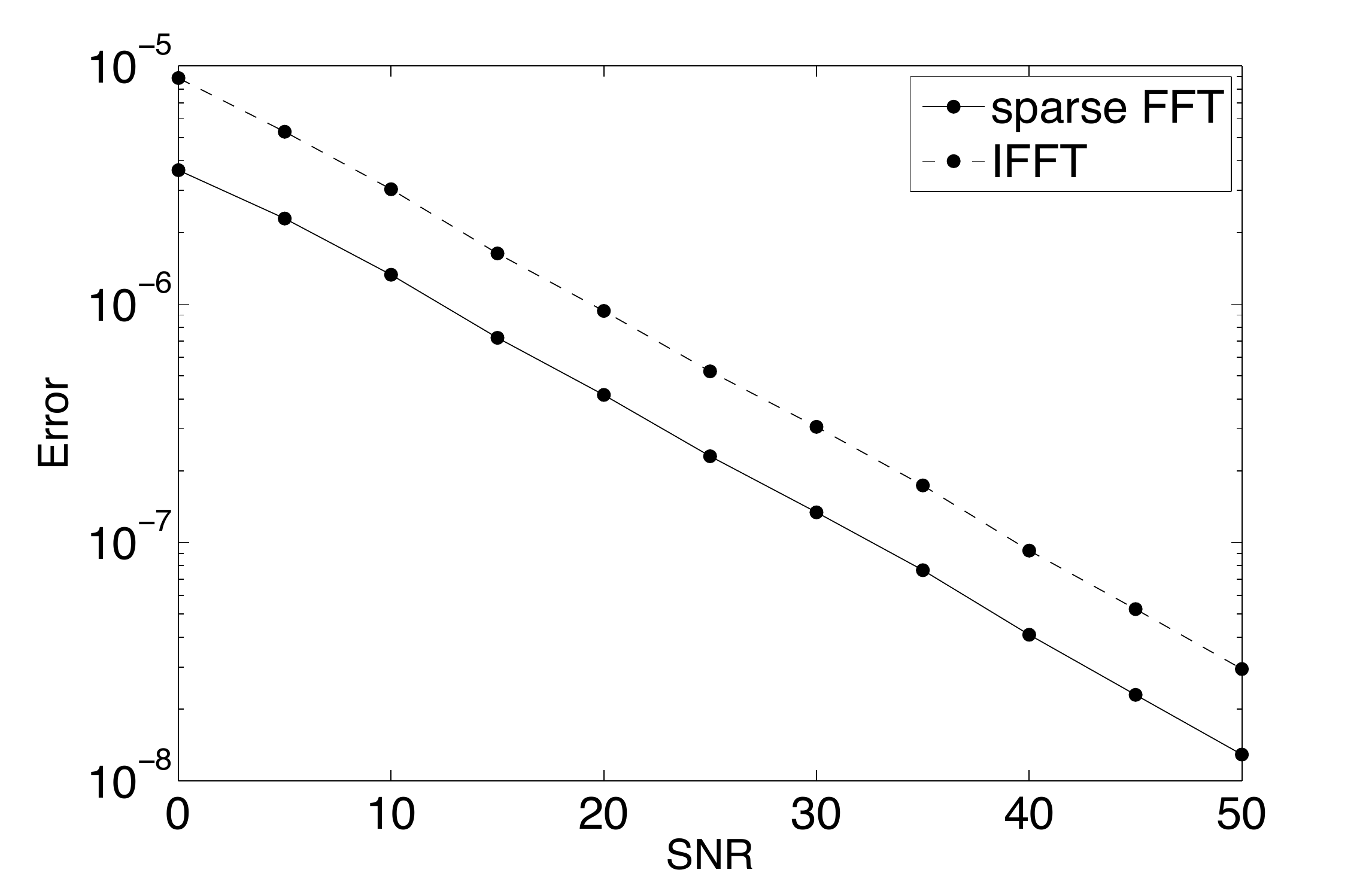}
\caption{Average reconstruction error $\|{\bf x}-{\bf x}'\|_2/N$ for different levels of uniform noise,\\ comparing our deterministic sparse FFT algorithm and inverse FFT.}
\label{algvsifft}
\end{center}
\end{figure}

Determining the first index of the support of $\bf x$ (i.e., finding the correct $\mu = \mu^{(J)}$) is one of the crucial points of the algorithm for noisy input data. If $\mu^{(L+1)}$ is not identified correctly, the correct support interval cannot be found anymore, even if all shifts in step 3 of Algorithm \ref{alg2} are correct. 

In a third experiment we again take randomly generated vectors $\bf x$ of length $N=2^{22}$ with support length $m=50$ or $m=2^{18}$, where  $|\text{Re}(x_k)|\le 10$, $|\text{Im}(x_k)|\le 10$ for $k$ in the support interval. 

\begin{table}[h]
\begin{center}
\begin{tabular}{|c||c|c|c||c|c|c|}
\hline
\multirow{3}{*}{SNR} & \multicolumn{3}{c||}{$N=2^{22}$, $m=50$} & \multicolumn{3}{c|}{$N=2^{22}$, $m=2^{18}$}\\\cline{2-7}
 & correctly & \multirow{2}{*}{$\|\eps\|_\infty$} & \multirow{2}{*}{$\|\eps\|_1/N$}  & correctly & \multirow{2}{*}{$\|\eps\|_\infty$} & \multirow{2}{*}{$\|\eps\|_1/N$} \\
 & identified $\mu$ & & & identified $\mu$ & &  \\\hline
0 & 86\% & 68.367 & 37.002 & 78\% & 4927.294 & 2666.891 \\\hline
5 & 97\% & 37.799 & 20.458 & 93\% & 2911.275 & 1575.695 \\\hline
10 & 99\% & 22.262 & 12.049 & 97\% & 1365.686 & 739.186 \\\hline
15 & 100\% & 12.559 & 6.798 & 100\% & 841.737 & 455.585 \\\hline
20 & 100\% & 6.751 & 3.654 & 100\% & 483.223 & 261.542 \\\hline
25 & 100\% & 3.796 & 2.055 & 100\% & 261.897 & 141.752 \\\hline
30 & 100\% & 2.147 & 1.162 & 100\% & 144.593 & 78.259 \\\hline
35 & 100\% & 1.242 & 0.672 & 100\% & 84.374 & 45.665 \\\hline
40 & 100\% & 0.668 & 0.362 & 100\% & 47.666 & 25.800 \\\hline
\end{tabular}
\smallskip

\caption{Percentage of correctly identified $\mu$ and average norm of noise in 100 randomly chosen\\ vectors for different noise levels, dependent on length $N$ and support length $m$ of $\bf x$.}
\label{mudiff}
\end{center}
\end{table}

Table \ref{mudiff} shows 
the number of cases in which the first support index $\mu$ has been determined correctly by Algorithm \ref{alg2} for 100 randomly chosen vectors for each noise level. Additionally, we give an average for the norm of the noise vectors $\eps$ at each noise level. 

The results show that the algorithm succeeds for very short support intervals as well as for long support intervals compared to the full vector length. In cases where $\mu$ could not be determined correctly, the error originates from step 2 of the algorithm where $\mu^{(L+1)}$ has to be determined. However, the deviation from the correct support was small ($\le 6$) in any case.

We can conclude that even for high noise level, the support of the reconstructed  vector ${\bf x}$ is correctly found in most of the cases. The support is always correct for noise levels  with SNR $\ge 15$. Table \ref{mudiff} also shows that the absolute noise $\varepsilon_{k}$ at each component can be considerably large. It is essentially larger for  vectors with larger support  since in this case also the signal energy grows accordingly.

\subsection*{Acknowledgement}

The research in this paper is partially funded by the project PL 170/16-1 of the German Research Foundation (DFG). This is gratefully acknowledged.


\begin{thebibliography}{18}

\bibitem{Aka10}
A. Akavia, Deterministic sparse Fourier approximation  via fooling arithmetic progressions,
in Proc. 23rd COLT, 2010, pp. 381--393.

\bibitem{Aka14}
A. Akavia, Deterministic sparse Fourier approximation  via approximating arithmetic progressions,
IEEE Trans. Inform. Theory {\bf 60}(3) (2014), 1733--1741.

\bibitem{GIIS14}
A. Gilbert, P. Indyk, M.A. Iwen, and L. Schmidt,
Recent developments in the sparse Fourier transform,
IEEE Signal Processing Magazine {\bf 31}(5) (2014), 91--100.

\bibitem{Gr99}
J.J. Green, Calculating the maximum modulus of a polynomial  using Ste\v{c}kin's Lemma, SIAM J. Numer. Anal. {\bf 36}(4) (1999), 1022-1029.

\bibitem{HIKP12a}
H. Hassanieh, P. Indyk, D. Katabi, and E. Price, Near-optimal algorithm  for  sparse Fourier transform, Proc. 44th annual ACM symposium on Theory of Computing, 2012, 
pp. 563--578.


\bibitem{HIKP12b}
H. Hassanieh, P. Indyk, D. Katabi, and E. Price, Simple and practical algorithm for 
sparse Fourier transform,  Proc. 23th Annual ACM-SIAM Symposium on Discrete Algorithms (SODA '12), 2012, 
pp. 1183--1194.


\bibitem{HKPV13}
S. Heider, S. Kunis, D. Potts, and M. Veit, A sparse Prony FFT, Proc. 10th International Conference on Sampling Theory and Applications (SAMPTA), 2013, pp. 572--575.


\bibitem{IKP14} P. Indyk, M. Kapralov, and E. Price, (Nearly) sample-optimal Fourier transform, Proc.
25th Annual ACM-SIAM Symposium on Discrete Algorithms (SODA 14), 2014,
pp. 480--499.

\bibitem{I10} M.A. Iwen, Combinatorial sublinear-time Fourier algorithms,  Found. Comput. Math. {\bf 10} (2010), 303--338.

\bibitem{I13}
M.A.  Iwen, Improved approximation guarantees for sublinear-time Fourier algorithms, 
Appl. Comput. Harmon. Anal. {\bf 34} (2013), 57--82.

\bibitem{LWC13}
D. Lawlor, Y. Wang, and A. Christlieb,  Adaptive sub-linear time Fourier algorithms, Advances in Adaptive Data Analysis {\bf 5}(1)
(2013), 1350003 (25 pages).

\bibitem{PR13}
S. Pawar and K. Ramchandran, Computing  a $k$-sparse $n$-length discrete Fourier transform  using at most  $4k$ samples  and ${\cal O}(k \log k)$ complexity, IEEE International Symposium on Information Theory (2013),  pp. 464--468.

\bibitem{PT14}
G. Plonka and M. Tasche, Prony methods for recovery of structured functions,
GAMM-Mitt. {\bf 37}(2) (2014),  239--258. 

\bibitem{Sch13}
J. Schumacher, High performance sparse fast Fourier transform, Master Thesis, Computer Science. ETH Z\"urich, Switzerland, 2013.


\end{thebibliography}

\end{document}